\documentclass[12pt,reqno]{amsart}

\usepackage[arrow,matrix,curve]{xy}

\usepackage[dvips]{graphicx}

\usepackage{amssymb, latexsym, amsmath, amscd, array, hyperref
}

\newtheorem{theorem}{Theorem}[section]
\newtheorem{lemma}[theorem]{Lemma}
\newtheorem{proposition}[theorem]{Proposition}
\newtheorem{corollary}[theorem]{Corollary}

\theoremstyle{definition}
\newtheorem{definition}[theorem]{Definition}

\numberwithin{equation}{section}

\newcommand\R {{\mathbb R}}

\author[K. K.]{Karin U. Katz}\address{K. Katz, Department of
Mathematics, Bar Ilan University, Ramat Gan 52900
Israel}\email{katzmik@math.biu.ac.il}

\author[M. K.]{Mikhail G. Katz} \address{M. Katz, Department of
Mathematics, Bar Ilan University, Ramat Gan 52900
Israel}\email{katzmik@macs.biu.ac.il}

\author[D. K.]{Dmitry Kerner}\address{D. Kerner, Department of
Mathematics, Ben Gurion University of the Negev, Be'er Sheva 8410501
Israel}\email{dmitry.kerner@gmail.com}

\author[Y. L.]{Yevgeny Liokumovich} \address{Y. Liokumovich, Imperial
College, London}\email{y.liokumovich@imperial.ac.uk}

\begin{document}

\thispagestyle{empty}


\title {Determinantal variety and normal embedding}



\begin{abstract}
The space~$\emph{GL}^+_n$ of matrices of positive determinant inherits
an extrinsic metric space structure from~${\mathbb R}^{n^2}$. On the
other hand, taking the infimum of the lengths of all paths connecting
a pair of points in~$\emph{GL}^+_n$ gives an intrinsic metric.  We
prove bilipschitz equivalence between intrinsic and extrinsic metrics
on~${\it GL}^+_n$, exploiting the conical structure of the
stratification of the space of~$n\times n$ matrices by rank.
\end{abstract}

\maketitle


\section{Introduction}

Consider the group~${\it GL}^+(n,\R)$ of~$n\times n$ matrices of
positive determinant. It is an open submanifold of 
the space~$\R^{n^2}$ of~$n\times n$ matrices.  
Here~${\it GL}^+(n,\R)$ carries two metrics: its extrinsic ambient
Euclidean metric with distance function~$d_{ext}$, and the induced
intrinsic metric (i.e. least length of path) with distance
function~$d_{int}$.  Our main result is the following.
\begin{theorem}
\label{tm:main}
There is a constant~$C = C(n)$ such that~$d_{int}<C d_{ext}$.
\end{theorem}
Thus the determinantal variety is \emph{normally embedded}; in other
words, the intrinsic and extrinsic metrics are \emph{bilipschitz
equivalent}.  We prove Theorem\;\ref{tm:main} first for~$n=2,3$ and
then for general~$n$. The proof uses in an essential way the conical
structure of the stratification of
the space~$\R^{n^2}$ of~$n\times n$ matrices
by rank.  Indeed, the extrinsic and intrinsic metrics on a set as
simple as~$\{(x,y) \colon \,x^2-y^3 > 0\}$ are clearly inequivalent,
due to the fact that the curve~$x^2 - y^3 = 0$ has a cusp.


Bilipschitz equivalence has been studied by a number of authors; see
e.g., \cite{Bi16}, \cite{BM}, \cite{KK}, \cite{HR}.  For a study of
the Lipschitz condition in an infinitesimal context see \cite{KKN},
\cite{NK}.

\section{Solution in  dimension 2}
\label{s2}

For~$2\times 2$ matrices with coordinates given by the entries
$(a_{ij})$, the condition~$a_{11}a_{22}-a_{12}a_{21}=0$ for the
determinantal variety in the space of matrices translates into
\[
(a_{11}+a_{22})^2-(a_{11}-a_{22})^2
-(a_{12}+a_{21})^2+(a_{12}-a_{21})^2=0
\]
or to simplify notation
\begin{equation}
\label{e21}
x^2+y^2=z^2+w^2.
\end{equation}

\begin{definition}
Let~$X_2=\{A\in \R^{2^2} \colon \det(A)=0\}$ denote the determinantal
variety.
\end{definition}

\begin{lemma}
\label{l21}
The intersection~$X_2\cap S^3$ with the unit sphere
$S^3\subseteq\R^{2^2}$ is a flat~$2$-torus, namely the Clifford
torus~$T\subseteq S^3$.
\end{lemma}

Indeed, by~\eqref{e21} the complement~$S^3\setminus T$ is a union of
two solid tori, consisting respectively of matrices of positive and
negative determinant.  Note that~$X_2$ is a linear cone
over~$T\subseteq\R^{2^2}$.

\begin{lemma}
The metrics~$d_{int}$ and~$d_{ext}$ on~$T\subseteq\R^{2^2}$ are
bilipschitz equivalent.
\end{lemma}

\begin{proof}
By rescaling, we can assume that the furthest of the two points is at
unit distance from the origin.  By compactness, the only problem for a
pair of points~$p,q$ can arise when the points collide, i.e.,~$d(p,q)$
tends to~$0$.  By the smoothness of the Clifford torus, one
has~$\frac{d_{int}(p,q)}{d_{ext}(p,q)}\to 1$ as~$d(p,q)\to 0$.
\end{proof}

\begin{lemma}
\label{l22}
Bilipschitz equivalence holds for the intrinsic and the extrinstic
metrics on~$X_2$.
\end{lemma}

\begin{proof}
Let~$p,q\in X_2$.  If the apex~$O\in X_2$ is one of the points~$p,q$
then the intrinsic and the Euclidean distances between them coincide
by linearity of the cone.

Thus we can assume that~$p\not=O$ and~$q\not= O$.  To connect a pair
of points~$p,q,$ at different levels in the cone by a path lying in
the cone, assume without loss of generality that~$p$ is further
than~$q$ from the apex of the cone.  We slide~$p$ along the ray toward
the apex until it reaches the level of~$q$, and then connect it to~$q$
by a shortest path contained in that level.  The length of the
combined path is clearly bilipschitz with the extrinsic distance in
the ambient~$\R^{2^2}$.
\end{proof}

\begin{proposition}
\label{p24}
The intrinsic and the ambient metrics on the manifold of~$2\times 2$
matrices of positive determinant are bilipschitz equivalent.
\end{proposition}

\begin{proof}
The closure of~${\it GL}^+(2,\R)$ is a linear cone over the solid
torus.  Meanwhile,~${\it GL}^+(2,\R)$ itself is the cone without the
apex over the \emph{interior} of the solid torus.  We consider a
straight line path in~$\R^{2^2}$ connecting a pair of points in~${\it
GL}^+(2,\R)$.  The length of this path is the extrinsic distance by
definition.

The path does not necessarily lie entirely inside~${\it GL}^+(2,\R)$.
As the path is straight and the determinantal variety~$X_2$ is
algebraic, the path splits into a finite number of segments satisfying
the following:
\begin{enumerate}
\item
the interior of each segment lies fully either
inside~$\emph{GL}^+(2,\R)$ or inside the component~${\it GL}^-(2,\R)$;
\item
the endpoints are in~$X_2$.
\end{enumerate}
Applying Lemma~\ref{l22}, we replace every segment that lies
in the component~$\emph{GL}^-(2,\R)$ by an arc that lies in~$X_2$.

Finally, we push out the arc in~$X_2$ into~$GL^+(2,\R)$, i.e., replace
it by a nearby arc ``just inside''~$\emph{GL}^{+}(2,\R)$.  We will
explain the push-out procedure in detail since a similar argument will
be used in the general case.

The vertex~$O\in X_2$ is the unique singular point of the cone, i.e.,
the complement~$X_2\setminus \{ O\}$ is a smooth manifold.  Hence a
path in~$X_2$ disjoint from~$O$ can be pushed out infinitesimally
into~$\emph{GL}^+(2,\R)$ by following the normal direction, without
significantly affecting its length.

Suppose a path joining~$P,Q\in X_2$ passes through~$O$.  Using the
linear structure of the cone, the path can be replaced by the shorter
path given by the union of the straight line segments~$PO\cup
OQ\subseteq X_2$.  Next,~$P$ and~$Q$ can be replaced by nearby
points~$P',Q'\in \emph{GL}^{+}(2,\R)$.  We now form a path~$P'O\cup
OQ'$ lying entirely within~$\emph{GL}^{+}(2,\R)$ except for the single
point~$O$.

Let~$p\in O\!P'$ be the point of intersection of the segment with a
sphere~$S(O,\epsilon)$ of small radius~$\epsilon>0$ centered at~$O$,
and similarly for point~$q\in OQ'$.  The
intersection~$S(O,\epsilon)\cap\emph{GL}^{+}(2,\R)$ is an open solid
torus and therefore connected.  Hence there is a short path
\[
\gamma\subseteq S(O,\epsilon)\cap\emph{GL}^{+}(2,\R)
\]
joining~$p$ to~$q$.  The resulting path~$P'p\cup\gamma\cup qQ'$ is
only slightly longer than the original path joining~$P$ and~$Q$.  This
completes the proof of the bilipschitz property
for~$\emph{GL}^{+}(2,\R)$.
\end{proof}

\section{Singularities in dimension 3}
\label{s3}

For~$3\times 3$ matrices, the singular locus of the determinantal
variety~$X_3=\{A\in \R^{3^2}\colon \det A=0\}$ consists of matrices of
rank~$\leq 1$.  The variety~$X_3$ can be stratified as follows:
\[
X_3=X_{3,0} \cup X_{3,1}\cup X_{3,2}.
\]
Here the stratum~$X_{3,i}$ consists of matrices of rank~$i$.  The
stratum~$X_{3,1}$ lies in the closure of~$X_{3,2}$, while~$X_{3,0}$
lies in the closure of~$X_{3,1}$.  Each~$X_{3,i}$ is smooth.
Here~$X_{3,2}$ is of codimension~$1$ in the space of matrices,
while~$X_{3,1}$ is of codimension~$4$ in the space of matrices,
and~$X_{3,0}$ is a single point.

The closure~$\overline{X}_{3,2}$ is a cone on a smooth manifold away
from~$X_{3,1}$.  Thus the only potential obstruction to bilipschitz
equivalence is the singularity of~$\overline{X}_{3,2}$
along~$\overline{X}_{3,1},$ which we now analyze.
\begin{lemma}
\label{l31}
The~$5$-dimensional closure~$\overline{X}_{3,1}$ is a linear cone over
the smooth compact~$4$-manifold~$(S^2\times S^2)/\{\pm1\}$.
\end{lemma}
\begin{proof}
Here~$S^2 \times S^2 / \{\pm1\}$ parametrizes the intersection of
$X_{3,1}$ with the unit~$8$-sphere in the space of matrices.  The
manifold~$S^2\times S^2$ can be parametrized by pairs of unit column
vectors~$v,w\in S^2\subseteq\R^3$ producing a rank~$1$ matrix
\[
v\;{}^t\!w\in X_{3,1}.
\]
The element~$-1$ acts simultaneously on both factors of~$S^2\times
S^2$ by the antipodal map.
\end{proof}

Fix an element~$x \in X_{3,1}$.  We would like to understand the
bilipschitz property for~$X_3$ in a neighborhood of~$x$.  Exploiting
the action of~$\emph{GL}(3,\R)\times \emph{GL}(3,\R)$ on~$X_{3,1}$ by
right and left multiplication, we may assume that
\[
x=1\oplus \begin{pmatrix}0&0\\0&0\end{pmatrix}.
\]

\begin{definition}
\label{d32}
Let
\begin{equation*}
L_x = \{1\oplus A \colon A \in \R^{2^2}\} \subseteq \R^{3^2}
\end{equation*}
be a transverse slice to~$X_{3,1}\subseteq\R^{3^2}$ at~$x$, so that
\begin{equation}
\label{e31}
L_x \cap \overline X_{3,2} = \{1 \oplus A \in L_x \colon \det(A) =
0\}.
\end{equation}
\end{definition}

Note that of course transversality is weaker than orthogonality.  The
ambient Euclidean metric on~$\R^9$ plays no role here.

\begin{lemma}\label{l32}
A transverse slice for~$X_{3,1}\subseteq \overline{X}_{3,2}$ in local
coordinates is a linear cone over a Clifford torus.
\end{lemma}
\begin{proof}
The defining equation~\eqref{e31} of the slice is~$\det(A)=0$.  By
Lemma~\ref{l21} this is a cone over a torus.
\end{proof}
Now let~$U_T \subseteq \R^5$ be the unit ball. Choose open
sets~$U_N\subseteq L_x$ and~$U \subseteq \R^{3^2}$, each
containing~$x$, and a bilipschitz diffeomorphism
\begin{equation}
\label{e32}
\phi \colon U \to U_T \times U_N
\end{equation}
such that~$\phi(U\cap \overline X_{3,1}) = U_T \times\{x\}$
and~$\phi(U \cap \overline X_{3,2}) = U_T \times (U_N \cap
\overline{X}_{3,2}).$ This can be done using the~$GL(3,\R) \times
GL(3,\R)$ action, as explained in~\cite[p.\;138]{MP}.  Hence
Lemmas~\ref{l32} and~\ref{l22} imply that the intrinsic and extrinsic
metrics on~$U\cap \overline{X}_{3,2}$ are bilipschitz equivalent.

\begin{corollary}
\label{c33}
The intrinsic and extrinsic metrics on~$\overline X_{3,2} = X_3$ are
bilipschitz equivalent.
\end{corollary}

This follows from Lemma~\ref{l31} by a compactness argument.

\begin{theorem}
\label{s35}
The intrinsic and extrinsic metrics on~\it{GL}$^+(3,\R)$ are
bilipschitz equivalent.
\end{theorem}
\begin{proof}
The extrinstic distance between a pair of matrices of positive
determinant is the length of the straight line path in~$\R^{3^2}$
joining them.  We partition the path into finitely many segments,
where the interior of each segment lies entirely in a connected
component of~${\it GL}(3,\R)$ while the endpoints are
in~$X_3\subseteq\R^{3^2}$.  Then we apply Corollary~\ref{c33} to
replace each segment belonging to the component~$\emph{GL}^-(3,\R)$ by
an arc in~$X_3$.  If the arc in~$X_3$ lies in the smooth part
$X_{3,2}\subseteq X_3$ then it can be pushed out into
$\emph{GL}^+(3,\R)$ by a small deformation in the normal direction, as
in Section~\ref{s2}.

Unlike the case of Proposition~\ref{p24}, the determinantal variety is
not a cone on a manifold, so that an additional argument is required
to push the arc out of~$X_3$ and into~$GL_3^+$ while retaining
bilipschitz control.

If an arc in~$X_3$ joining points~$P,Q$ passes through the apex~$O\in
X_3$ then it can be replaced by the union~$PO\cup OQ$ and pushed out
into~$\emph{GL}^+(3,\R)$ as in the proof of Proposition~\ref{p24}.

Otherwise we exploit the local product structure
on~$X_3\setminus\{O\}$ as in~\eqref{e32}.  A path in~$X=X_3$ that dips
into the singular locus~$X_{3,1}\subseteq X$ can be handled as
follows.  Given a path~$\gamma \colon [0,1]\to X$, let~$a\in [0,1]$ be
the least parameter value~$t$ such that~$\gamma(t)$ is contained in
the singular locus~$X_{3,1}\subseteq X$, and~$b\in [0,1]$ the greatest
such value.

Step 1.  Consider the restriction of the path~$\gamma$
to~$[a,b]\subseteq[0,1]$.  We replace it by a path that lies entirely
in the singular locus~$X_{3,1}$.  Section~\ref{s4b} proves
that~$X_{3,1}$ is embedded in~$X$ in a bilipschitz fashion.  Hence
this replacement can be performed in a bilipschitz-controlled way.

Step 2.  Once the path~$\gamma([a,b])$ is in the singular locus, we
exploit a local trivialisation to push it in a constant direction as
follows.  Choose~$\delta>0$ sufficiently small to be specified later.
Then~$\gamma(a-\delta)$ and~$\gamma(b+\delta)$ are in the smooth part
$X_{3,2}\subseteq X$.  The path can easily be shortened to a smooth
one, still denoted~$\gamma$.  Over a sufficiently small neightborhood
of the smooth path, we can choose a smooth (and in particular
bilipschitz) trivialisation of~$X$ over the path~$\gamma([a,b])$.  Let
\begin{equation}
\label{s32}
I_{0,1}(t)= I_{0,1} =\begin{pmatrix} 1& 0 \\ 0& 0
\end{pmatrix}
\end{equation}
be a constant section of the bundle over the path.  Relative to the
trivialisation of the bundle we can form a new path 
\[
\bar\gamma_{\epsilon}(t)=\gamma(t)+\epsilon I_{0,1}(t)
\]
which pushes~$\gamma([a,b])$ in the constant direction \eqref{s32} for
each~$t$.  By construction, the new path is contained in the
nonsingular part~$X_{3,2}$ of the determinantal variety,
where~$\epsilon>0$ is chosen small enough so that the length of the
path stays close to the original length of the path in~$X_{3,1}$.

Step 3.  It remains to check that the path can be patched up with the
value of the path~$\gamma$ at the parameter values~$a-\delta$
and~$b+\delta$.  Since the determinantal variety in the fiber is a
cone over a connected space, the two values can be connected by an
arbitrarily short path, provided~$\epsilon$ and~$\delta$ are chosen
small enough.

Step 4.  Once the path lies in the nonsingular part~$X_{3,2}\subseteq
X$ of the determinantal variety, it can be pushed out into
$\emph{GL}^+(3,\R)$ by following the normal direction as in the case
$n=2$.
\end{proof}

The higher dimensional case is treated inductively in
Section~\ref{s5}.

\section{The general determinantal variety}
\label{s4b}

Let~$X_{n,k}\subseteq \R^{n^2}$ be the stratum of rank~$k$ matrices.
Thus~$X_n$ is the closure of~$X_{n,n-1}$ and each~$X_{n,k}$ is a
smooth connected manifold of dimension~$n^2 - (n-k)^2.$ For each~$x
\in X_{n,k}$ choose an open set~$U_x \subseteq \R^{n^2}$
containing~$x$, a metric ball~$U^T_{x} \subseteq \R^{n^2-(n-k)^2},$ an
open set~$U^N_x\subseteq \R^{(n-k)^2}$ containing~$0,$ and a
bilipschitz diffeomorphism~$\phi=\phi_x \colon U_x \to U^T_{x} \times
U^N_x,$ such that
\begin{equation}
\label{e41}
\phi(U_x \cap X_{n,k}) = U^T_{x} \times \{0\}, \quad \phi(U_x \cap
X_n) = U^T_{x} \times (U^N_x \cap X_{n-k}).
\end{equation}
See \cite[p.\;138]{MP} for the construction of the
diffeomorphism~$\phi_x$.  To ensure the bilipschitz property, it
suffices to shrink slightly~$U_x$,~$U^N_x$ and~$U^T_{x}$.

Assume by induction that we have proved the bilipschitz equivalence of
the intrinsic and extrinsic metrics on~$X_\ell$ for~$\ell < n$.
Generalizing Lemma~\ref{l31} to dimension~$n,$ we see~$\overline
X_{n,1}$ is a linear cone over a compact smooth submanifold. So, the
intrinsic and extrinsic metrics on~$\overline X_{n,1}$ are bilipschitz
equivalent. For each~$x \in X_{n,1},$ the bilipschitz diffeomorphism
$\phi_x$ and the induction hypothesis for~$X_{n-1}$ show the intrinsic
and extrinsic metrics are bilipschitz equivalent on~$U_x \cap X_n.$
Then, a compactness argument shows the bilipschitz property holds for
a sufficiently small open neighborhood~$U_\epsilon\! \left(\,\overline
X_{n,1}\right)$ of~$\overline X_{n,1}$ in~$X_n$.  More
precisely,~$U_\epsilon$ is defined to be a cone over an
$\epsilon$-neighborhood in the unit sphere.

Next we consider the stratum~$X_{n,2}$. The complement~$X_{n,2}
\setminus\, U_\epsilon\! \left(\, \overline X_{n,1} \right)$ is a cone
on a compact smooth manifold with boundary.  For each
point~$x\in{}X_{n,2}\setminus\,U_\epsilon \!\left(\, \overline
X_{n,1}\right)$, the bilipschitz diffeomorphism~$\phi_x$ and the
induction hypothesis for~$X_{n-2}$ show the intrinsic and extrinsic
metrics are bilipschitz equivalent on~$U_x \cap X_n$.  Arguing by
compactness as before, we obtain the bilipschitz property for a
sufficiently small neighborhood~$U_\epsilon\! \left(\,\overline
X_{n,2}\right)$ of~$\overline X_{n,2}$ in~$X_n$.  Here\,~$\epsilon$\,
may have to be chosen smaller than the one chosen for~$X_{n,1}$.

We proceed in this way until we obtain the bilipschitz property for a
neighborhood of~$X_{n,n-2}$ in the determinantal variety~$X_n$.  By
compactness, the bilipschitz property holds for~$X_n$ itself.

\section
{Bilipschitz property for the set of matrices of positive determinant}
\label{s5}

We prove Theorem\;\ref{tm:main} by pushing a path in the determinantal
variety out into the component~$\emph{GL}^+(n,\R)$, and mimicking the
proof of Theorem~\ref{s35}.

If the path~$\gamma$ passes via the apex~$O\in X$, it can be replaced
by a pair of straight line segments and pushed out into~$C$ as in
Section~\ref{s2}.

Otherwise let~$k\geq1$ be the least rank of a matrix~$\gamma(t)$ along
the path, and let~$a_k\leq b_k\in[0,1]$ be respectively the first and
last occurrences of a matrix of rank~$k$.  As in Section~\ref{s3}, we
push the path~$\gamma([a_k,b_k])$ into~$X_{n,k}\subseteq X$, by
applying the results of Section~\ref{s4b}.  We then push it out into
$X_{n,k+1}$ by following a constant direction, and patch it up at the
endpoints as in Step 3 of the proof of Theorem~\ref{s35}.

The new path is in~$X_{n,k+1}$.  We now choose the corresponding
parameter values~$a_{k+1}, b_{k+1}\in [0,1]$ and proceed inductively.
Thus the path can be pushed out into~$X_{n,n-1}$.  Finally we follow
the normal direction to push the path out into~$\emph{GL}^+(n,\R)$.

\section*{Acknowledgments}
M.~Katz was partially funded by the Israel Science Foundation grant
no.~1517/12.  D.~Kerner was partially supported by the Israel Science
Foundation grant 844/14.  This paper answers a question posed by Asaf
Shachar at MO%
\footnote{See \url{http://mathoverflow.net/questions/222162}}
and we thank him for posing the question.  We are grateful to Yves
Cornulier for a helpful comment posted there.  We are grateful to Jake
Solomon for providing the proof in Section~\ref{s4b} of the general
case of the bilipschitz property for the determinantal variety.  We
thank Jason Starr for a helpful comment posted at MO.%
\footnote{See \url{http://mathoverflow.net/questions/230668}}
We thank Amitai Yuval for pointing out a gap in an earlier version of
the article, and Alik Nabutovsky and Kobi Peterzil for useful
suggestions.

 \end{document}